\documentclass[11pt]{article}

\usepackage{tgpagella}
\usepackage[utf8]{inputenc}
\usepackage[T1]{fontenc}
\usepackage{amssymb}
\usepackage{amsfonts}
\usepackage{amsmath}
\usepackage{amsthm}
\usepackage{mathrsfs}
\usepackage{natbib}
\usepackage{comment}
\usepackage{xcolor}
\usepackage{ulem}
\usepackage{hyperref}

\newtheorem{theorem}{Theorem}
\newtheorem{proposition}[theorem]{Proposition}
\newtheorem{lemma}[theorem]{Lemma}

\newtheorem{corollary}[theorem]{Corollary}

\theoremstyle{definition}
\newtheorem{definition}[theorem]{Definition}
\newtheorem{convention}[theorem]{Convention}

\newtheorem{remark}[theorem]{Remark}

\newcommand{\PA}{\textnormal{PA}}

\newcommand{\set}[2]{\lbrace #1 \ \mid \ #2 \rbrace }
\newcommand{\res}{\upharpoonright}
\newcommand{\CT}{\textnormal{CT}}
\newcommand{\df}[1]{\textbf{#1}}
\newcommand{\num}[1]{\underline{#1}}

\newcommand{\INT}{\textnormal{INT}}

\newcommand{\ElDiag}{\textnormal{ElDiag}}
\newcommand{\val}[1]{{#1}^{\circ}}

\newcommand{\tuple}[1]{\langle #1 \rangle}
\renewcommand{\Pr}{\textnormal{Pr}}

\newcommand{\dom}{\textnormal{dom}}

\newcommand{\FV}{\textnormal{FV}}
\newcommand{\qcr}[1]{\ulcorner #1 \urcorner}

\newcommand{\lh}{\textnormal{lh}}

\newcommand{\FinSeq}{\textnormal{FinSeq}}
\newcommand{\SeqOInd}{\textnormal{SOInd}}
\newcommand{\SeqInd}{\textnormal{SInd}}
\newcommand{\PropRef}{\textnormal{PropRef}}
\newcommand{\PropSnd}{\textnormal{PropSnd}}
\newcommand{\Prop}{\textnormal{Prop}}
\newcommand{\GR}{\textnormal{GRef}}
\newcommand{\DC}{\textnormal{DC}}
\newcommand{\DCin}{\textnormal{DC-in}}
\newcommand{\DCout}{\textnormal{DC-out}}
\newcommand{\QFC}{\textnormal{QFC}}

\newcommand{\qfSent}{\textnormal{qfSent}}
\newcommand{\Tr}{\textnormal{Tr}}
\newcommand{\IDelta}{\textnormal{I}\Delta}

\newcommand{\REG}{\textnormal{REG}}

\newcommand{\LPA}{\mathscr{L}_{\PA}}
\newcommand{\form}{\textnormal{Form}}
\newcommand{\formSeq}{\textnormal{FormSeq}}
\newcommand{\Term}{\textnormal{Term}}
\newcommand{\Sent}{\textnormal{Sent}}

\newcommand{\ClTerm}{\textnormal{ClTerm}}

\newcommand{\ClTermSeq}{\textnormal{ClTermSeq}}

\newcommand{\SentSeq}{\textnormal{SentSeq}}
\newcommand{\SentSet}{\textnormal{SentSet}}
\newcommand{\Comp}{\textnormal{Comp}}
\newcommand{\Asn}{\textnormal{Asn}}
\newcommand{\Var}{\textnormal{Var}}

\title{The two halves of disjunctive correctness} 
\author{Cezary Cieśliński\footnote{Faculty of Philosophy, University of Warsaw.}, Mateusz Łełyk\footnote{Faculty of Philosophy, University of Warsaw.}, Bartosz Wcisło\footnote{Institute of Mathematics, Polish Academy of Sciences.}}

\begin{document}
	
	\maketitle

\begin{abstract}
	Ali Enayat had asked whether two halves of Disjunctive Correctness ($\DC$) for the compositional truth predicate are conservative over Peano Arithmetic. In this article, we show that the principle ``every true disjunction has a true disjunct'' is equivalent to bounded induction for the compositional truth predicate and thus it is not conservative. On the other hand, the converse implication ``any disjunction with a true disjunct is true'' can be conservatively added to $\PA$. The methods introduced here allow us to give a direct nonconservativeness proof for $\DC$.
\end{abstract}	

\section{Introduction}

The area of axiomatic truth theories analyses the concept of truth by studying first-order theories which try to capture various properties of this notion. These theories are formulated as follows: We choose a base theory strong enough to represent syntax  (this is typically Peano Arithmetic, $\PA$). To this theory, we add a fresh predicate $T(x)$ whose intended reading is ``$x$ is (a code of) a true sentence'' together with axioms governing the behaviour of that predicate. A notable example of such a theory is $\CT^-$ which stipulates that the truth predicate satisfies Tarski's compositional conditions for arithmetical sentences. For instance, a conjunction $\phi \wedge \psi$ is true iff both $\phi$ and $\psi$ are true. 

 If we add to $\CT^-$ full induction for the formulae containing the truth predicate, the resulting theory, called $\CT$, is not conservative over $\PA$, that is, it proves arithmetical theorems which cannot be demonstrated in $\PA$ itself. More specifically, we can show by induction on the lengths of proofs that every sentence provable in $\PA$ is true (a principle called global reflection over $\PA$) and, consequently, that $\PA$ is consistent.\footnote{This is one of the basic results of truth theory. A comprehensive introduction to the area, including this result, can be found in \citep{halbach}.} 
On the other hand, by a theorem of Kotlarski, Krajewski, and Lachlan (see \citep{kkl}) $\CT^-$ itself does not prove any new arithmetical theorems. 

In \citep{cies}, it was shown that already $\CT^-$ extended with a principle of propositional reflection ``sentences derived in propositional logic from true premises are true'' suffices to prove global reflection over $\PA$. In other words, an overtly nonconservative principle can be derived without explicitly assuming induction on a ground of a rather innocuous principle of overtly truth-theoretic nature. Subsequently, it turned out that a number of other truth-theoretic principles are not conservative over $\PA$ and are all equivalent to $\Delta_0$-induction for the compositional truth predicate ($\CT^-$ with $\Delta_0$-induction is called $\CT_0$).\footnote{ For the proof of nonconservativity of $\CT_0$, see \citep{wcislyk}.} The picture that emerged was that there seems to be a ``minimal natural'' nonconservative extension of $\CT^-$. 

The ``dividing line'' between conservative and nonconservative extensions of $\PA$ has been named ``Tarski boundary'' by Ali Enayat.\footnote{One can find more information on the Tarski boundary in \citep{cies_ksiazka}. A more concise discussion is also contained in \citep{lelyk_wcislo_studia} and \citep{lelyk_studia}.} One of the most striking results on Tarski boundary was obtained by  \cite{EnayatPakhomov}. It was shown that among theories equivalent to $\CT_0$ is $\CT^-$ together with the principle of disjunctive correctness, $\DC$, which states that any finite (but possibly nostandard) disjunction is true iff one of the disjuncts is true. This axiom appears to be a mild, natural, extension of the compositional clauses and yet turns out to carry the full strength of $\Delta_0$-induction. 

It was not clear whether this result can be pushed further in the following manner: Disjunctive correctness can be naturally split into two halves. The first half is a principle $\DCin$ saying ``a disjunction with true disjunct is true'' and the second is $\DCout$ which says ``a true disjunction has a true disjunct''. It has been asked in \citep{EnayatPakhomov} whether the second  of these principles can be added conservatively to $\PA$.\footnote{See Question 5.3 in \citep{EnayatPakhomov}. The conservativity of $\DCin$ was settled in December 2018 and stated in the formulation of that question, but the proof was not published.} 

In this article, we analyse both of the above principles. We show that over $\CT^-$, $\DCin$ gives rise to a conservative extension of $\PA$ while $\DCout$ yields $\Delta_0$-induction. The methods used for the nonconservativeness result yield a new, direct proof that $\DC$ is yet another incarnation of $\Delta_0$-induction for the formulae containing the truth predicate.

\section{Preliminaries} \label{sec_prelim}

We consider truth theories over Peano arithmetic, $\PA$, as our base theory.\footnote{The choice is motivated mostly by a certain tradition in the field. However, the results make sense and still hold true over much weaker theories, say $\IDelta_0 + \exp$.} It is an axiomatic theory in a language $\LPA = \{S, +, \times,0\}$ whose axioms consist of inductive definitions of addition and multiplication in terms of the successor function $S$ together with the full induction scheme. Although this theory overtly speaks of natural numbers, in fact it is strong enough to capture objects such as finite sets, finite sequences, or finite graphs. Crucially, $\PA$ is capable of expressing syntactic notions such as ``term'', ``formula'' or ``proof'' and proving basic facts about these notions such as ``a conjunction of two formulae is a formula.'' We assume that the reader is familiar with the coding of syntax and the basic metamathematics. This is discussed in many sources such as \citep{kaye}, Chapter 9 or \citep{hajekpudlak}, Chapter I, Section 1(d), pp.50--61. Throughout the article, we will explain those bits of notation that seem not to explain themselves. We provide a glossary of all formalised notions we are using throughout the paper in the appendix. 

\begin{convention} \label{conv_syntax} \
	\begin{itemize}
		\item We will often conflate G\"odel numbers of syntactic objects with those objects. We will also use formulae denoting syntactic expressions as if they were sets. For instance, we will write $\phi \in \Sent_{\LPA}$ rather than $\Sent_{\LPA}(\phi)$, where $\Sent_{\LPA}$ expresses that $\phi$ is (a G\"odel code of) an arithmetical sentence.
		\item 	Provably functional formulae will be used as if denoting actual functions. For instance, we will write $\FV(\phi)$ for the set of free variables of $\phi$.
		\item  In particular, we will often write the results of syntactic operations without explicitly mentioning the operations themselves. For instance, if $\phi, \psi \in \form_{\LPA}$, we will freely speak of the conjunction $\phi  \wedge \psi$ rather than ``the only $z$ such that $z$ is a conjunction of $\phi$, $\psi$.''
	\end{itemize}
\end{convention}
The research on Tarski boundary concerns extensions of the compositional truth theory. Let us define it. By writing $\num{n}$, we mean (the G\"odel code of) a \df{numeral} denoting the number $n$, that is,
\begin{displaymath}
	\qcr{\underbrace{S(S\ldots S(0) \ldots)}_{\textnormal{``$S$'' repeated $n$ times}}}.
\end{displaymath}
If $t$ is (a G\"odel code of) a closed arithmetical term, then by $\val{t}$ we mean the value of that term. Thus we have, for instance,
\begin{displaymath}
	\mathbb{N} \models \val{\qcr{S0 + S(S(0))}} = \val{\num{3}} = 3 .
\end{displaymath}
Note that a value of an arithmetical term can be computed in a primitive recursive way and the natural function computing it may be formalised in $\PA$. Now we can proceed to the actual definition.

\begin{definition} \label{def_CTminus}
	By $\CT^-$ (Compositional Truth) we mean a theory in the arithmetical language with a fresh unary predicate $T$ extending $\PA$ with the following axioms:
	\begin{itemize}
		\item $\forall x \ \Big( T(x) \rightarrow \Sent_{\LPA}(x)\Big).$
		\item $\forall s,t \in \ClTerm_{\LPA} \Big(T(s=t) \equiv \val{s} = \val{t}\Big).$
		\item $\forall \phi \in \Sent_{\LPA} \Big(T(\neg \phi) \equiv \neg T \phi\Big).$
		\item $\forall \phi,\psi \in \Sent_{\LPA} \Big(T(\phi \vee \psi ) \equiv T\phi \vee T \psi\Big).$ 
		\item $\forall \phi,\psi \in \Sent_{\LPA} \Big(T(\phi \wedge \psi ) \equiv T\phi \wedge T \psi\Big).$ 
		\item $\forall \phi \in \form_{\LPA} \forall v \in \Var \Big(\Sent_{\LPA}(\exists v \phi) \rightarrow  T(\exists v \phi) \equiv \exists x \ T\phi(\num{x})\Big).$
		\item $\forall \phi \in \form_{\LPA} \forall v \in \Var \Big(\Sent_{\LPA}(\forall v \phi) \rightarrow  T(\forall v \phi) \equiv \forall x \ T\phi(\num{x})\Big).$
		\item $\forall \phi \in \form_{\LPA} \forall \bar{s}, \bar{t} \in \ClTermSeq_{\LPA} \Big(\val{\bar{s}} = \val{\bar{t}} \rightarrow T\phi(\bar{s}) = T\phi(\bar{t}) \Big).$
	\end{itemize}
	Above, $\ClTermSeq_{\LPA}(x)$ is a formula expressing ``$x$ is a sequence of closed arithmetical terms.'' The rest of the notation should be self-explanatory and, as we remarked earlier, it is discussed in the appendix. 
	
	 By $\CT$, we mean $\CT^-$ with full induction for the extended language. By $\CT_0$, we mean $\CT^-$ with induction for $\Delta_0$ formulae containing the truth predicate. (Note that already $\CT^-$ contains full arithmetical induction.) 
\end{definition}

The last clause in the above axioms for $\CT^-$ is called \df{Regularity Axiom}, $\REG$. It is not included among the basic axioms in the standard presentations of this theory, like \citep{halbach} or \citep{cies_ksiazka}. The version of $\CT^-$ with $\REG$ appears, e.g., in \citep{enayatlelykwcislo} or \citep{lelyk_wcislo_local_collection}. Admittedly, the regularity axiom has a clearly different status than the rest of the axioms of $\CT^-$. We add it mostly for two (admittedly technical) reasons which we will explain in Section \ref{sec_regularity}.

Although $\CT^-$ seems to express the crucial properties of the truth predicate, it is not arithmetically stronger than $\PA$. The result was essentially proved by \cite{kkl}.\footnote{The original result concerns satisfaction classes over a purely relational language. This work has been extended in \citep{kaye} to languages with terms and in \citep{engstrom_thesis} to truth classes over languages with terms. Subsequently, \cite{enayatvisser2} introduced a new, elegant and flexible method of constructing satisfaction classes which allowed them to strengthen the previous results in a number of interesting ways. They worked again in purely relational languages. A version for functional languages can be found in \citep{cies_ksiazka}. A version of Enayat--Visser construction covering languages with functional symbols with the regularity axioms included is discussed in \citep{lelyk_wcislo_local_collection}. A reader may find a proof of a stronger result also in this article in Section \ref{sect_dcin}.
}
\begin{theorem}[Kotlarski--Krajewski--Lachlan] \label{tw_kkl}
	$\CT^-$ is conservative over $\PA$, i.e., for any sentence $\phi \in \LPA$ if $ \CT^- \vdash \phi$, then $\PA \vdash \phi$. 
\end{theorem}
The above result contrasts with the situation for $\CT$. By induction on the length of proofs, we can show that every proof in $\PA$ has true conclusion. In other words, we can prove in $\CT$ the following principle of \df{Global Reflection}, $\GR$:
\begin{displaymath}
	\forall \phi \in \Sent_{\LPA} \Big( \Pr_{\PA}(\phi) \rightarrow T \phi \Big).
\end{displaymath}
Notice that $\CT^- + \GR$ is clearly not conservative over $\PA$, since in particular $\neg T (0 \neq 0)$ is provable in the compositional theory $\CT^-$ and, by contraposition $\neg \Pr_{\PA}(0 \neq 0 )$. The latter is not provable in $\PA$ by G\"odel's second theorem. 

As we have already mentioned, a number of seemingly unrelated principles turned out to be equivalent to this  canonical nonconservative axiom $\GR$. One of them is $\Delta_0$-induction for the compositional truth predicate. Another principle which will play an important role in this article is \df{Propositional Reflection}, $\PropRef$, defined as follows:
\begin{displaymath}
	\forall \phi \in \Sent_{\LPA} \Big(\Pr^{T}_{\Prop}(\phi) \rightarrow T \phi \Big),
\end{displaymath}
where $\Pr_{\Prop}^{T}(x)$ means that $x$ is derivable in propositional logic from the set of premises $\Gamma$, such that $T(y)$ holds for each $y \in \Gamma$. As we have already noted, \cite{cies} showed that over $\CT^-$, $\PropRef$ is equivalent to $\Delta_0$ induction for the truth predicate. Subsequently, $\CT_0$ was shown by \cite{wcislyk} to be arithmetically equivalent to $\GR$ and then, in \cite{lelyk_thesis}, to be exactly the same theory as $\CT^- + \GR$. Another presentation of the last result can be also found in a recent preprint \citep{lelyk_global_reflection}.

Related to propositional reflection is the following principle of \df{Propositional Soundness}, $\PropSnd$: 
\begin{displaymath}
	\forall \phi \in \Sent_{\LPA} \Big(\Pr_{\Prop}(\phi) \rightarrow T\phi \Big).
\end{displaymath}
In effect, this axiom expresses that any arithmetical sentence which is a propositional tautology is true. It is still unknown whether $\CT^- + \PropSnd$ is conservative over $\PA$. In the next section, we shall present a partial result towards this problem. 

Now, let us turn to the main subject of our article. If $M \models \PA$ and $\bar{\phi} \in \SentSeq_{\LPA}(M)$ is a coded sequence of sentences, we can form their disjunction $\bigvee \bar{\phi}$, which we will also denote by $\bigvee_{i \leq c} \phi_i$ if the length of $\bar{\phi}$ is $c$. We always assume that in ``big disjunctions'' parentheses are grouped to the left, so that the following equality holds in $M$:
\begin{displaymath}
	\bigvee_{i \leq c+1} \phi_i = \bigvee_{i \leq c} \phi_i \vee \phi_{c+1}.
\end{displaymath}
In effect, $\bigvee_{i \leq c} \phi$ denotes the following formula:
\begin{displaymath}
	(((\phi_0 \vee \phi_1) \vee \ldots )\vee \phi_{c-1}) \vee \phi_c.
\end{displaymath}

The precise definition of how disjunctions over multiple disjuncts are parenthesised can actually matter in some cases. In the presence of $\Delta_0$-induction for the extended language, one can show that any two disjunctions with the same disjuncts are equivalent, no matter how the disjuncts are ordered and grouped together. However, this is not the case in pure $\CT^-$. The precise parenthesising can become relevant, since it dictates the relations between disjunctions over arbitrarily many disjuncts and the usual binary disjunctions. We will return to this issue in Section \ref{sect_parentheses}.

By \df{Disjunctive Correctness}, $\DC$, we mean the following axiom:
\begin{displaymath}
	\forall \bar{\phi} \in \SentSeq_{\LPA} \bigg(T \left( \bigvee \bar{\phi} \right)\equiv \exists i \leq \lh(\bar{\phi}) \ T \phi_i \bigg).
\end{displaymath}
Therefore, $\DC$ states that a finite (but possibly nonstandard) disjunction is true iff it has a true disjunct. Very surprisingly, $\DC$ is yet another incarnation of $\CT_0$ as shown by \cite{EnayatPakhomov}. 
\begin{theorem}[Enayat--Pakhomov] \label{th_dc_equiv_ct0}
	$\CT^- + \DC$ and $\CT_0$ are equivalent. 
\end{theorem}
As we mentioned, $\CT_0$ is not conservative over $\PA$. Hence the following fact easily follows:
\begin{corollary}[Enayat--Pakhomov] \label{cor_dc_not_conservative}
	$\CT^- + \DC$ is not conservative over $\PA$. 
\end{corollary}
Disjunctive correctness can be naturally split into two principles. The first is $\DCout$ (``a true disjunction has a true disjunct''):
\begin{displaymath}
	\forall \bar{\phi} \in \SentSeq_{\LPA} \Big(T\bigvee \bar{\phi} \rightarrow \exists i \leq \lh(\bar{\phi}) \ T \phi_i \Big).
\end{displaymath}
The second is $\DCin$ (``a disjunction with a true disjunct is true''):
\begin{displaymath}
	\forall \bar{\phi} \in \SentSeq_{\LPA} \Big( \exists i \leq \lh(\bar{\phi}) \  T \phi_i \rightarrow T\bigvee \bar{\phi} \Big).
\end{displaymath}
As we have mentioned, the status of $\DCout$ was not settled in \citep{EnayatPakhomov} and the claim that $\DCin$ is conservative was stated without a proof. The subsequent parts of this article will be devoted to  the analysis of both principles. 

Let us finish this section by summing up the positive results on equivalences of theories of truth relevant for this work:
\begin{theorem} \label{th_many_faces}
	The following theories are equivalent:
	\begin{enumerate}
		\item $\CT_0$.
		\item $\CT^- + \GR$.
		\item $\CT^- + \PropRef$.
		\item $\CT^-  + \DC$. 
	\end{enumerate}
\end{theorem}

\section{Yablo sequences and disjunctive correctness} \label{sec_dcout}

In this section, we prove that $\DCout$ is equivalent to $\CT_0$. The argument is indirectly inspired by the classical Visser--Yablo paradox: \footnote{See \citep{yablo}.} 
Consider the sequence of sentences $Y_n, n \in \mathbb{N}$, such that $Y_n$ says: ``some sentence $Y_k$ for $k > n$ is false.'' If for some $k$, $Y_k$ is false, then every sentence $Y_l$ for $l>k$ is true. However, if $Y_l$ is true, then there is some $m>l$ such that $Y_m$ is false contradicting the assumption. Hence, every sentence in the Yablo sequence is true. However, if some sequence in the Yablo sequence is true, then some has to be false, so they cannot be all true. We reach a contradiction.

Using a construction inspired by the Yablo sequence, we will show that $\DC$-out implies $\CT_0$. We will actually use two intermediate principles which are overtly related to $\Delta_0$-induction. By \df{sequential induction}, $\SeqInd$, we mean the following axiom:\footnote{A related principle of \df{Modus Ponens correctness} was introduced earlier by Ali Enayat in an unpublished note \citep{enayat_fine_tuning}. The principle states that if a conjunction of the implications $\phi_{i} \rightarrow \phi_{i+1}$ is true for $i = 0,1, \ldots, c$ and $\phi_0$ is true, then $\phi_{c+1}$ is true. However, as stated, this principle is conservative over $\PA$. Namely, it holds in any model in which all conjunctions of nonstandard length are false and by a construction very similar to that presented in Section \ref{sect_dcin}, for any completion $U$ of $\PA$, we can find models with this property satisfying $U$.}
\begin{displaymath}
	\forall s \in \FinSeq \ \Big(Ts_0 \wedge \forall i < \lh(s) - 1 \big(T s_i \rightarrow T s_{i+1} \big) \rightarrow \forall j < \lh(s)  \ Ts_j \Big)
\end{displaymath}
The \df{sequential order induction}, $\SeqOInd$, is a natural variant of the above principle:
\begin{displaymath}
	\forall s \in \FinSeq \ \Big(\forall j < \lh(s) \big((\forall i<j Ts_i) \rightarrow T s_j \big) \rightarrow \forall l < \lh(s)\  Ts_l \Big).
\end{displaymath}
As we already mentioned, $\SeqInd$ and $\SeqOInd$ are clearly related to $\Delta_0$-induction:
\begin{proposition} \label{prop_seqoind_equiv_ct0}
	$\CT^- + \SeqOInd$ and $\CT_0$ are equivalent.	
\end{proposition}

\begin{proof}
	$\CT_0$ clearly entails $\SeqOInd$. On the other hand, observe that $\SeqOInd$ implies $\PropRef$ which by Theorem \ref{th_many_faces} is equivalent to $\CT_0$. Working in $\CT^- + \SeqOInd$, fix a proof $(\phi_0, \ldots, \phi_c)$ in propositional logic from true premises. We can assume that the proof system is chosen so that for all $i$, either $\phi_i$ is an assumption of the proof or  $\phi_i$ is obtained by modus ponens from two formulas appearing earlier in the proof. In particular, for every $j$, if every formula $\phi_j$ for $j< i$ is true, then $\phi_i$ is true which, by $\SeqOInd$, implies that the conclusion of the proof is true. Thus $\PropRef$ holds.  
\end{proof}
Now, we get to the core argument of the article:
\begin{theorem} \label{th_dcout_implies_sind}
	Over $\CT^-$, $\DCout$ implies $\SeqInd$.
\end{theorem}
\begin{proof}
	Working in $\CT^- + \DCout$, fix any sequence $(\phi_0, \ldots, \phi_c) \in \SentSeq_{\LPA}$ such that $T\phi_0$ holds and for each $i$, $T\phi_i$ entails $T\phi_{i+1}$. Let us define a sequence $\psi_i, i \leq c$ as follows:
	\begin{eqnarray*}
		\psi_0 & : = & \phi_0 \\
		\psi_{j+1} & : = & \neg \phi_{j+1} \rightarrow \bigvee_{i < j+1} \neg \psi_i.
	\end{eqnarray*}
	
	We claim that for all $j \leq c$, the sentence $\psi_j$ is true. Suppose that $\neg T \psi_j$ holds for some $j$. Then  $j>0$, so we have:
	\begin{displaymath}
		\neg T \phi_j \wedge \neg T \bigvee_{i<j} \neg \psi_i.
	\end{displaymath}
	By compositional conditions, this is equivalent to:
	\begin{displaymath}
		\neg T \phi_j \wedge  \neg T \left( \bigvee_{i<j-1} \neg \psi_i \right) \wedge T \psi_{j-1}.
	\end{displaymath}
	Expanding the definition of $\psi_{j-1}$, we obtain:
	\begin{displaymath}
		\neg T\phi_j \wedge \neg T \left( \bigvee_{i< j-1} \neg \psi_i \right) \wedge \left (T \phi_{j-1} \vee T \left( \bigvee_{i<j-1} \neg \psi_i \right) \right).
	\end{displaymath}
	However, $\neg T \phi_j$ implies $\neg T \phi_{j-1}$. Therefore, we have the following:
	\begin{displaymath}
		\neg T\phi_j \wedge  \neg T \left( \bigvee_{i< j-1} \neg \psi_i \right)  \wedge  T \left( \bigvee_{i<j-1} \neg \psi_i \right). 
	\end{displaymath}
	This contradiction concludes the proof of the claim. Notice that the proof of the claim only uses the fact that provably in $\CT^-$, 
	\begin{displaymath}
		T \left( \bigvee_{i \leq j+1} \eta_i \right) \equiv T \left( \bigvee_{i \leq j} \eta_i\right) \vee T\eta_{j+1}.
	\end{displaymath}
	
	Now, we show that for all $j<c$, $\phi_j$ is true. Suppose otherwise and fix $j$ such that $\neg T\phi_j$. Since $T\psi_j$ holds, we have:
	\begin{displaymath}
		T \bigvee_{i< j} \neg \psi_i
	\end{displaymath}
	By $\DCout$, we can fix $i<j$ such that $\neg T \psi_i$ holds. However, this contradicts the previous claim.
\end{proof}

As an application of the above result, we show that $\DCout$ is the same theory as $\DC$. In particular, $\CT^- + \DCout$ is not conservative over $\PA$.

\begin{theorem} \label{th_seqind_implies_dcin}
	$\CT^- + \SeqInd$ implies $\DCin$. Consequently, $\DCout$ and $\DC$ are equivalent over $\CT^-$.  
\end{theorem}
\begin{proof}
	Working in $\CT^- + \SeqInd$ fix any sequence $(\alpha_0, \ldots, \alpha_c)$  such that $\alpha_i \in \Sent_{\LPA}$ for all $i \leq c$. Suppose that $\alpha_j$ is true for some $j \leq c$ and notice that 
	\begin{displaymath}
		\bigvee_{i\leq j} \alpha_i = \bigvee_{i \leq j-1} \alpha_i \vee \alpha_j ,
	\end{displaymath}
	hence the disjunction of $\alpha_i$ up to and including $j$ has to be true as well. Moreover, for any $k$, if $\bigvee_{i \leq k} \alpha_i$ is true, then $\bigvee_{i \leq k+1} \alpha_i$ is true. Hence by sequential induction (starting with $\bigvee_{i \leq j} \alpha_i$), $T\bigvee_{i \leq c} \alpha_i$ holds. 
\end{proof}

Another application is a more perspicuous proof of nonconservativity of $\CT^- + \DC$. 
\begin{theorem} \label{th_dc_implies_seqoind}
	$\CT^- + \DC$ implies $\SeqOInd$. Hence, $\CT^- + \DC$ is equivalent to $\CT_0$. 	
\end{theorem}
\begin{proof}
	Working in $\CT^- + \DC$, fix any sequence $(\phi_0, \ldots, \phi_c)$ and suppose that for any $j$, if $T\phi_i$ holds for all $i< j$, then $T\phi_j$ holds. By $\DC$, the following implication holds for all $j \leq c$:
	\begin{displaymath}
		T \neg \bigvee_{i \leq j} \neg \phi_i \rightarrow T \neg \bigvee_{i \leq j+1} \neg \phi_i.
	\end{displaymath}
	(Of course, these are essentially big conjunctions, but \textit{prima facie}, disjunctive correctness does not imply conjunctive correctness.) Then, by $\SeqInd$, we conclude that 
	\begin{displaymath}
		T \neg \bigvee_{i \leq j} \neg \phi_i
	\end{displaymath}
	holds for each $i$ which, again using $\DC$ implies that $T\phi_j$ holds for each $j$. 
	
	By Proposition \ref{prop_seqoind_equiv_ct0} it follows that $\CT_0 \subseteq \CT^- + \DC$. Since $\DC$ is easily provable in $\CT_0$, we conclude that $\CT^- + \DC$ is equivalent to $\CT_0$.
\end{proof}

\begin{remark} \label{rem_dc_via_int}
	The main nonconservativeness proof for $\DC$ in \citep{EnayatPakhomov} consists of two parts: it is first shown that over $\CT^-$, $\DC$ implies the axiom of internal induction $\INT$, to be defined in the next section, and then a much more direct proof that $\DC + \INT$ implies $\CT_0$ follows. It is easy to verify that $\CT^- + \SeqInd$ entails internal induction and since we know that $\CT^- + \DC$ implies $\SeqInd$, we obtain  a still different proof that $\CT^- + \DC$ is equivalent to $\CT_0$.
\end{remark}

We can present the above results in a slightly more abstract manner. This will allow us to obtain a significantly simpler proof of the result from \citep{wcislo_prop_plus_qf} on the strength of certain extensions of Propositional Soundness.

\begin{definition} \label{defi_outer_disjunction}
	Let $U$ be a theory extending $\CT^-$. We say that $U$ has \df{outer disjunctions} if there exists a provably functional formula $D(x,y)$ such that for any $\bar{\phi} = \tuple{\phi_1,\ldots,\phi_c} \in \SentSeq_{\LPA}$, we have $D(\bar{\phi}) \in \Sent_{\LPA}$\ and the following two properties hold provably in $U$:
	\begin{itemize}
		\item $\forall \bar{\phi} \in \SentSeq_{\LPA} \forall \psi \in \Sent_{\LPA} \ \Big(T D(\bar{\phi}\frown\tuple{\psi}) \equiv TD(\bar{\phi}) \vee T\psi \Big).$
		\item $\forall \bar{\phi} \in \SentSeq_{\LPA} \Big(TD(\bar{\phi}) \rightarrow \exists i \leq \lh(\bar{\phi}) \ T\phi_i \Big).$
	\end{itemize}	
	We call $D$ as above an outer disjunction of $\phi_1, \ldots, \phi_c$. 
\end{definition}
In other words, a theory of truth has outer disjunctions if it has some uniform construction that behaves like disjunctions in $\CT^- + \DCout$. 
\begin{proposition} \label{prop_outer_disjunctions}
	Suppose that $U$ has outer disjunctions. Then it satisfies $\SeqOInd$ and in particular, it contains $\CT_0$. 
\end{proposition}
\begin{proof}[Sketch of the proof.]
	This is literally the same argument as in Theorems \ref{th_dcout_implies_sind}, \ref{th_seqind_implies_dcin}, and \ref{th_dc_implies_seqoind}. We only used the fact that $\CT^- + \DCout$ has outer disjunctions. 
\end{proof}
By \df{quantifier-free correctness}, $\QFC$, we mean the following axiom:
\begin{displaymath}
	\forall \phi \in \qfSent_{\LPA} \Big(\Tr_0(\phi) \rightarrow T\phi\Big).
\end{displaymath}
Let us  pause for a moment and explain the notation. The formula $\qfSent_{\LPA}(x)$ expresses that $x$ is a quantifier-free sentence of $\LPA$, i.e., a Boolean combination of closed term equations. Peano arithmetic has a canonical way of deciding whether such (possibly nonstandard) sentences should be true or false by applying partial arithmetical truth predicates. The formula $\Tr_0$ is such a predicate for $\Delta_0$ formulae.\footnote{For a more detailed explanation of what partial arithmetical truth predicates are, consult \citep{kaye}, Section 9 or \citep{hajekpudlak}, Chapter I, Section 1(d).} It can be checked that $\QFC$ does not bring any arithmetical strength to $\CT^-$. 
\begin{proposition} \label{prop_qfc_conservative}
	$\CT^- + \QFC$ is conservative over $\PA$.
\end{proposition}
The above proposition follows by a routine application of the methods introduced by Enayat and Visser, see e.g. \citep{enayatvisser2}. A proof of this result in the exact same setting in which we work (with regularity axioms included in the definition of $\CT^-$) can be found in \citep{wcislo_prop_plus_qf}. In the same article, it was shown with a different argument using a propositional construction called disjunctions with stopping conditions that $\QFC$ becomes significantly stronger when combined with $\PropSnd$. Now we can prove that result with a simpler argument:

\begin{proposition} \label{prop_qfc_plus_prop_has_outer_dijunctions}
	$\CT^- + \QFC + \PropSnd$ has outer disjunctions. In particular, it is equivalent to $\CT_0$. 
\end{proposition}
\begin{proof}
	The ``in particular'' part follows by Proposition \ref{prop_outer_disjunctions} and the fact that $\QFC$ and $\PropSnd$ are clearly implied by $\Delta_0$-induction. So it is enough to show  that $\CT^- + \QFC + \PropSnd$ has outer disjunctions. 
	
	For $\bar{\phi} = (\phi_1, \ldots, \phi_c)$, let 
	\begin{displaymath}
		D(\bar{\phi}) = \exists x \leq c \bigvee_{i \leq c} \num{i} = x \wedge \phi_i.
	\end{displaymath}
	Let us check that $D$ is an outer disjunction. First, consider the formula $D(\bar{\phi}\frown \tuple{\phi_{c+1}})$, i.e.
	\begin{displaymath}
		\exists x \leq c+1 \bigvee_{i \leq c+1} \num{i} = x \wedge \phi_i.
	\end{displaymath}
	We want to check that it is true iff either $D(\bar{\phi})$ is true or $\phi_{c+1}$ is true. Suppose that 
	\begin{displaymath}
		T\exists x \leq c+1 \ \bigvee_{i \leq c+1} \num{i} = x\wedge \phi_i.
	\end{displaymath}
	By the compositional clauses and the definition of big disjunctions, this implies:
	\begin{displaymath}
		\exists x \leq c+1 \  T \left( \bigvee_{i \leq c} \num{i} = \num{x} \wedge \phi_i \right) \vee \exists x \leq c+1 \ T\left( \num{c+1} = \num{x} \wedge \phi_{c+1}\right).
	\end{displaymath}
	It can be easily checked by compositional clauses that if the second clause holds, then $T \phi_{c+1}$ holds. So it is enough to check that the first clause in fact implies 
	\begin{displaymath}
		\exists x \leq c \  T \left( \bigvee_{i \leq c} \num{i} = \num{x} \wedge \phi_i \right).
	\end{displaymath} 
	To this end, we have to check that 
	\begin{displaymath}
		\neg T\left( \bigvee_{i \leq c} \num{i} = \num{c+1} \wedge \phi_i \right).
	\end{displaymath}
	However, notice that the following implication is an instance of a propositional tautology:
	\begin{displaymath}
		\left( \bigwedge_{i \leq c}  \num{i} \neq \num{c+1} \right) \rightarrow \neg \left( \bigvee_{i \leq c} \num{i} = \num{c+1} \wedge \phi_i \right).
	\end{displaymath}
	Hence, by $\PropRef$, the whole implication is true and by $\QFC$ the antecedent is true as well and the conclusion follows by compositionality. This ends the proof of the implication.
	
	The verification that $TD(\bar{\phi}) \vee T\phi_{c+1}$ implies $TD(\bar{\phi} \frown \tuple{\phi_{c+1}})$ is similar, but simpler, as it only uses compositional clauses of $\CT^-$. 
	
	Now, suppose that the following holds:
	\begin{displaymath}
		T\exists x \leq c \bigvee_{i \leq c} \num{i} = x \wedge \phi_i.
	\end{displaymath}
	We want to check that for some $i \leq c$, $T\phi_i$ holds. By compositional clauses, we know that there exists $a \leq c$ such that:
	\begin{displaymath}
		T \bigvee_{i \leq c} \num{i} = \num{a} \wedge \phi_i.
	\end{displaymath}
	Again, using $\QFC$ and $\PropSnd$, we check that the following holds:
	\begin{displaymath}
		T \left( \neg \phi_a \rightarrow \neg \bigvee_{i \leq c} \num{i} = \num{a} \wedge \phi_i \right).
	\end{displaymath}
	Thus we can conclude that $T\phi_a$ holds which ends the proof. 
\end{proof}
\begin{remark}
	In the above proof, we could actually show that the formula $D(\bar{\phi})$ satisfies both directions of $\DC$. 
\end{remark}
The principle $\SeqInd$ clearly looks very related to $\Delta_0$-induction for the truth predicate. Indeed, it turns out that over $\CT^-$ the two principles are equivalent.

\begin{theorem} \label{th_sind_has_outer_disjunction}
	$\CT^- + \SeqInd$ has outer disjunctions. 
\end{theorem}
\begin{proof}
	For a coded sequence $(\phi_1, \ldots, \phi_k)$ of $\LPA$-sentences, let $\bigwedge_{i \leq c} \phi_i$ be their conjunction with parentheses grouped to the left so that we have:
	\begin{displaymath}
		\bigwedge_{i \leq c+1} \phi_i = \bigwedge_{i \leq c} \phi_i \wedge \phi_{c+1}.
	\end{displaymath}
	Using $\SeqInd$, we can show that if $T\phi_i$ holds for every $i \leq c$, then $T \bigwedge_{i \leq c} \phi_i$. We show this by considering an auxiliary sequence
	\begin{displaymath}
		\bigwedge_{i \leq 1} \phi_i, \bigwedge_{i \leq 2} \phi_i, \ldots, \bigwedge_{i \leq c} \phi_i.
	\end{displaymath}
	
	Now, we claim that the formula
	\begin{displaymath}
		D(\bar{\phi}) := \neg \bigvee_{i \leq c} \neg \phi_i
	\end{displaymath}
	is an outer disjunction. It is easy to verify that over  $\CT^-$:
	\begin{displaymath}
		T D(\bar{\phi} \frown \tuple{\phi_{c+1}}) \equiv T D(\bar{\phi}) \vee T\phi_{c+1}. 
	\end{displaymath}
	We show that $D$ satisfies the second condition of outer disjunctions ($\DCout$) by contraposition. If there is no $i \leq c$ such that $T\phi_i$, then $T \neg \phi_i$ holds for every $i$. Consequently, the following conjunction is true:
	\begin{displaymath}
		\bigwedge_{i \leq c} \neg \phi_i
	\end{displaymath}
	which implies that $D \bar{\phi}$ cannot be true. This shows that $D$ is indeed outer disjunction provably in $\CT^- + \SeqInd$. 
\end{proof}
Let us summarise the above results:
\begin{corollary} \label{cor_equivalences_sind_soind_dcout}
	The following theories are equivalent:
	\begin{itemize}
		\item $\CT_0$
		\item $\CT^- + \SeqInd$.
		\item $\CT^- + \SeqOInd$.
		\item $\CT^- + \DCout$.
	\end{itemize}
\end{corollary}

We conclude this section with some results that clarify the status of $\SeqInd$ and $\SeqOInd$. First of all, let us note that $\SeqInd$ and $\SeqOInd$ are really just some forms of induction axioms and as such they do not require us to assume $\CT^-$ to make sense. They both clearly follow from $\IDelta_0(T)$, the induction scheme for $\Delta_0$-formulae containing the predicate $T$ (which is now treated just as some arbitrary predicate). It turns out that these principles form a strict hierarchy.
\begin{theorem} \label{hierarchy_sind_soind_delta0}
	Over $\PA$, $\IDelta_0(T) \rightarrow \SeqOInd \rightarrow \SeqInd$. Moreover, none of the implication reverses.
\end{theorem}
\begin{proof}
	Both implications are straightforward, so let us show that neither reverses.
	
	$(\SeqOInd \nvdash \IDelta_0(T))$. Let $M$ be a countable nonstandard model of $\PA$. Let $s_i, i < \omega$ be an enumeration of all coded sequences in $M$. We will inductively construct a sequence of finite subsets of the standard cut $\omega^M$, $A_i,B_i, i < \omega$ such that for all $i$, $A_i \cap B_i = \emptyset$.  The sets $A$ are approximations to $T$, the sets $B$ are approximations to the complement.

	Let $A_0 = B_0 = \emptyset$. For an arbitrary $i < \omega$, we consider two cases: if all the values of $s_i$ are elements of $A_i$, i.e.:
	\begin{displaymath}
		\set{a \in M}{\exists j< \lh(s_i) \ a = s_i(j)} \subseteq A_i,
	\end{displaymath}
	then we set $A_{i+1} = A_i, B_{i+1} = B_i$. Otherwise, let $j$ be the least element such that $s_i(j) \notin A_{i}$ (this element exists, since $|A_i|$ is standard, and therefore, $A_i$ is arithmetically definable). Let $b = s_i(j)$, let $a \neq b$  be an arbitrary element  in  $\omega \setminus (A_i \cup B_i)$ and set:
	\begin{eqnarray*}
		A_{i+1} & = & A_i \cup \{a\}  \\
		B_{i+1} & = & B_i \cup \{b\}.
	\end{eqnarray*}
	
	Finally, we set $T = \bigcup_{i \in \omega} A_i$. We claim that $(M,T) \models \SeqOInd$, but $(M,T)$ does not satisfy $\IDelta_0(T).$ The latter claim holds, since by construction $T$ is an infinite (hence cofinal) subset of the standard cut $\omega$. To check that the first claim holds, fix any coded sequence $s$. Pick $i < \omega$ such that $s = s_i$. By construction, either all values of $s_i$ are in $A_i$, hence in $T$ or there exists an $l$ such that for all $j<l$, $s_i(j) \in A_{i+1}$, but $s_i(l) \in B_{i+1}$ (which means that it is not in $T$). This shows that $(M,T) \models \SeqOInd$.
	
	($\SeqInd \nvdash\SeqOInd$) As before, fix an arbitrary countable nonstandard model $M \models \PA$ and let $s_i, i <\omega$ be an (external) enumeration of coded sequences from $M$. We inductively construct two sequences of sets $A_i,B_i$ such that for all $i$, $A_i \cap B_i = \emptyset$,  the sets $B_i$ are finite, and the sets $A_i$ contain only finitely many nonstandard elements of $M$.
	
	We construct $A_i, B_i$ as follows: let $A_0 =  \omega^M$ (the standard initial cut of $M$) and $B_0 = \emptyset$. 
	
	For an arbitrary $i$, if the set of values of $s_i$ is (standardly) finite, we set $A_{i+1} = A_i, B_{i+1} = B_i$. Otherwise, we consider two further subcases. Suppose that the set of $j$  such that $s_i(j) \in B_i$ is an initial segment in $M$ (possibly empty). Let $j_0$ be its supremum (it exists by arithmetical induction, since $B_i$ is finite; if the considered set is empty, then by definition, its supremum is $0$). On the other hand, the set of values of $s_i$ is by assumption infinite, so by overspill there exists $j>j_0$ such that $s_i(j) \notin A_i$. Let  $a:= s_i(j-1)$ and let $b:= s_i(j)$.  
	
	If the set 	$\set{j < \lh(s)}{s_i(j) \in B_i}$	 is not an initial segment of $M$, then (again by arithmetical induction and finiteness of $B_i$) there exists some $j$ such that $s_i(j) \notin B_i$ and $s_{i}(j+1) \in B$. Let $a:=s_i(j), b:=s_i(j+1)$. 
	
	In both cases, set:
	\begin{eqnarray*}
		A_{i+1} & = & A_i \cup \{a \} \\
		B_{i+1} & = & B_i \cup \{b \}.
	\end{eqnarray*}
	Finally, let $T = \bigcup_{i \in \omega} A_i$. We claim that $(M,T) \models \SeqInd$, but not $\SeqOInd$. 
	
	To see that $(M,T) \models \SeqInd$, fix any coded sequence $s \in M$. If the set of values of $s_i$ has (standard) finite number of elements, then $T$ cannot violate the sequential induction axiom for $s$. So suppose that the number of values of $s$ is nonstandard. Fix $i$ such that $s = s_i$ in our enumeration. By construction, there exists $j \in M$ such that $s(j) \in A_{i+1} \subset T$ and $s(j+1) \in B_{i+1} \subset M \setminus T$. So it is not the case that for all $x$, $T(s(x)) \rightarrow T(s(x+1))$ and thus the sequential induction axiom for $s$ is satisfied. 
	
	On the other hand, $(M,T)$ does not satisfy $\SeqOInd$. Indeed, fix any $c \in M \setminus \omega$ and consider the identity sequence: $s(i) = i$ for $i =0,1, \ldots, c$. We claim that the sequential order induction fails for this sequence. Let us check that $s$ is progressive, i.e., if for any $j<i$ $s(j) \in T$, then $s(i) \in T$.  
	
	Fix any $i \leq c$. If $i \in \omega$, then $T(i)$ holds. On the other hand, if $i$ is nonstandard, then consider the sequence $t: = s \res i$. It has nonstandardly many values, all of which are strictly below $i$. By construction, there exists an element $b$ occurring in this sequence which is not in $T$. This $b$ witnesses that not all $j<i$ are in $T$. Hence $s$ is progressive. On the other hand, not all terms of $s$ are in $T$ which means that $\SeqOInd$ fails.

\end{proof}

\section{Conservativeness of $\DCin$} \label{sect_dcin}

In the previous section, we have shown that $\CT^- + \DCout$ is not conservative over $\PA$ and that, in fact, it is another incarnation of $\CT_0$. Now we will use methods introduced by \cite{enayatvisser2} to show that the related principle $\DCin$ is actually conservative over $\PA$. The Enayat--Visser technique typically allows us to combine various results, so we can show joint conservativity of several distinct principles. Here we will illustrate this point by requiring that the constructed truth predicate additionally satisfies internal induction. Let us recall that principle. If $(M,T) \models  \CT^-$, then each formula $\phi \in \form^{\leq 1}_{\LPA}(M)$ ``defines'' a set $\set{x \in M}{\phi(\num{x}) \in T}$. \df{Internal induction} $(\INT)$ expresses that each such set satisfies the induction principle:
\begin{displaymath}
	\forall \phi \in \form^{\leq 1}_{\LPA} \Big(T\phi(\num{0}) \wedge \forall x \big(T\phi(\num{x}) \rightarrow T \phi(\num{x+1})\big) \rightarrow \forall x \ T\phi(\num{x})\Big).
\end{displaymath} 
It was essentially observed already by \cite{kkl} that $\CT^- + \INT$ is conservative over $\PA$. This result can be also proved using cut elimination as in \citep{leigh} or the methods invented by Enayat and Visser which we will use in this section in order to show the following theorem:

\begin{theorem} \label{th_dcin_plus_int_conservative}
	$\CT^- + \DCin + \INT$ is conservative over $\PA$.
\end{theorem}
The theorem is a direct corollary to the following model-theoretic result:
\begin{theorem} \label{th_true_infinite_disjunctions}
	Let $M \models  \PA$. Then there exists $M' \succeq M$ and a $T' \subseteq M'$ such that $(M',T') \models \CT^- + \INT$ and for all sequences $\bar{\phi} \in \SentSeq_{\LPA}(M)$ of nonstandard length, the disjunction $\bigvee \bar{\phi}$ is in $T'$. 
\end{theorem}
In other words, we will construct a model in which every disjunction with infinitely many disjuncts is true. Such a model clearly satisfies $\DCin$: fix any sequence of sentences $\bar{\phi}$ and suppose that there is $i \leq \lh(\bar{\phi})$ such that $T'\phi_i$. If $\bar{\phi}$ has nonstandard length, then $\bigvee \bar{\phi} \in T'$ by assumption. On the other hand, if the length of $\bar{\phi}$ is standard, then $\bigvee \bar{\phi} \in T'$ directly by the compositional axioms. So it is enough to prove Theorem \ref{th_true_infinite_disjunctions}.

In the proof, we will construct an elementary sequence of models $M_i \models \PA$ and a sequence of satisfaction classes $S_i \subset M_i^2, i<\omega$ such that $S_{i+1}$ will satisfy compositional clauses for formulae in the model $M_i$. Let us first define what a satisfaction class actually is. Below, if $s$ is (a code of) a term and $\alpha$ is a function whose domain contains the free variables of $s$ (an $s$-assignment), by $s^{\alpha}$ we mean the formally computed value of $s$ under the valuation $\alpha$. For instance, if $\alpha$ ascribes the value $2$ to $x$ and $5$ to $y$, then 
\begin{displaymath}
	\mathbb{N} \models \qcr{SSx \times Sy}^{\alpha} = 24.
\end{displaymath}
If $\alpha, \beta$ are assignments and $v$ is a variable, then by $\beta \sim_v \alpha$ we mean that $\dom(\beta) \supseteq \dom(\alpha) \cup \{v\}$ and $\beta(w) = \alpha(w)$ for every $w \in \dom(\alpha) \setminus \{v\}$. In other words, $\beta$ is just like $\alpha$, possibly except for the value $\beta(v)$ which is not even required to be defined for $\alpha$. 

\begin{definition} \label{def_satisfaction}
	Let $M \models \PA$ and let $\phi \in \form_{\LPA}(M)$. By the \df{compositional clauses} for $\phi$, $\Comp(\phi)$, we mean the disjunction of the following sentences:
	\begin{enumerate}
		\item $\exists s,t \in \Term_{\LPA} \ \Big(\phi = (s=t) \wedge \forall \alpha \in \Asn(\phi) \ \Big(S(\phi,\alpha) \equiv s^{\alpha} = t^{\alpha} \Big) \Big).$
		\item $\exists \psi \in \form_{\LPA} \ \Big(\phi = (\neg \psi) \wedge \forall \alpha \in \Asn(\phi) \ \Big(S(\phi,\alpha) \equiv \neg S(\psi,\alpha)\Big)\Big).$
		\item $\exists \psi, \eta \in \form_{\LPA}(M) \ \Big(\phi = (\psi \vee \eta) \wedge \forall \alpha \in \Asn(\phi) \Big(S(\phi,\alpha) \equiv S(\psi,\alpha) \vee S(\eta,\alpha) \Big)\Big).$
		\item $\exists \psi, \eta \in \form_{\LPA}(M) \ \Big(\phi = (\psi \wedge \eta) \wedge \forall \alpha \in \Asn(\phi) \Big(S(\phi,\alpha) \equiv S(\psi,\alpha) \wedge S(\eta,\alpha) \Big)\Big).$
		\item  $\exists \psi \in \form_{\LPA}(M) \exists v \in \Var \ \Big(\phi = (\exists v \psi) \wedge \forall \alpha \in \Asn(\phi) \Big(S(\phi,\alpha) \equiv \exists \beta \sim_v \alpha S(\psi, \beta)\Big)\Big).$  
		\item  $\exists \psi \in \form_{\LPA}(M) \exists v \in \Var \ \Big(\phi = (\forall v \psi) \wedge \forall \alpha \in \Asn(\phi) \Big(S(\phi,\alpha) \equiv \forall \beta \sim_v \alpha \ S(\psi, \beta)\Big)\Big).$  
	\end{enumerate}

	We say that $S \subset M^2$ is a \df{satisfaction class} if there exists a subset $D  \subseteq \form_{\LPA}(M)$ such that the following holds:
	\begin{itemize}
		\item If $(\phi,\alpha) \in S$, then $\phi \in \form_{\LPA}(M)$ and $\alpha \in \Asn(\phi)$.
		\item If $(\phi,\alpha) \in S$ for some $\alpha \in \Asn(\phi)$, then $\Comp(\phi)$ holds. 
		\item If $(\phi,\alpha) \in S$, then $\phi \in D$ or $\phi = \neg \psi$ for some  $\psi \in D$.
		\item If $\phi \in D$, then $\Comp(\phi)$ holds.
		\item If $\phi \in D$ and $\psi$ is a direct subformula of $\phi$, then $\psi \in D$.
		\item If $\phi \in D$, then for all $\alpha \in \Asn(\phi)$, $(\phi,\alpha) \in S$ or $(\neg \phi,\alpha) \in S.$
	\end{itemize}
	By the \df{domain} of $S$, $\dom(S)$, we mean the maximal set $D$ satisfying the above conditions (the maximality requirement is needed, as there is a slight ambiguity in how to count negations of formulae which are not satisfied by any assignment). A satisfaction class is \df{full} iff its domain is the whole set $\form_{\LPA}(M)$. 
\end{definition}
The definition of a satisfaction class presented above involves some technical requirements which might seem to be slightly too restrictive, so let us briefly explain our motivations. The definition of a satisfaction class completely agrees with the usual one in the case of full satisfaction classes.  However, if $S$ is not a full satisfaction class, i.e., if it does not satisfy compositional conditions for all formulae, it becomes ambiguous whether we should interpret the fact that $(\phi,\alpha) \notin S$ as saying that $\phi$ is not satisfied under the valuation $\alpha$ or that $S$ simply does not decide $\phi$ which might be a technical nuisance in some proofs or statements of results. This includes results and arguments in this article, though the specific difficulties we overcome with this definition will not be really visible. Therefore, we essentially require that for every formula $\phi$ we can unambiguously tell whether it is decided by $S$ or not. This requirement is harmless: Define a pre-satisfaction class on $M$ as a set $S \subset M^2$ such that it satisfies compositional axioms on some set $D$ of formulae closed under direct  subformulae and does not contain any pair $(\phi,\alpha)$ for $\phi \notin D$. Then for any pre-satisfaction class, we can canonically define a satisfaction class extending it. Simply take the maximal set $D$ with the two mentioned properties and extend $S$ with all pairs $(\neg \phi,\alpha)$ such that $\phi \in D$ and $(\phi,\alpha) \notin S.$ One can check that after such a one-step extension, we obtain a satisfaction class in our sense. 

Satisfaction classes and truth predicates satisfying $\CT^-$ are very closely related objects. However, the link between them is not as direct as one could hope (see \citep{wcislo_definability_automorphisms} for a discussion of this connection). We have to introduce a certain technical condition in order to switch between them in a completely unproblematic manner.

Let $M \models \PA$, $\phi, \psi \in \form_{\LPA}(M)$, $\alpha \in \Asn(\phi), \beta \in \Asn(\psi)$. We say that the pairs $(\phi,\alpha), (\psi,\beta)$ are \df{extensionally equivalent}, $(\phi,\alpha) \simeq (\psi,\beta)$ if there exists a formula $\eta$ and two sequences of closed terms $\bar{s}, \bar{t} \in \ClTermSeq_{\LPA}$ of the same length such that the sequences of their values are equal and 
\begin{displaymath}
	\phi[\alpha] = \eta(\bar{s}), \psi[\beta] = \eta(\bar{t}),
\end{displaymath}
where $\phi[\alpha]$ is a sentence obtained by substituting in $\phi$ the numeral $\num{\alpha(v)}$ for each variable $v$. For instance, let $\phi = \exists x (x+y = SS0)$, $\psi = \exists x (x + u\times v = w + S0)$, and let  $\alpha \in \Asn(\phi), \beta \in \Asn(\psi)$ be such that $\alpha(y) = 2, \beta(u) = 2, \beta(v) = \beta(w) = 1$. Then $(\phi,\alpha) \simeq (\psi,\beta)$ as witnessed by the formula 
\begin{displaymath}
	\eta = \exists x (x + v_0 = v_1)
\end{displaymath}
and the terms $(SS0, SS0), (SS0\times S0,S0 + S0)$.
Finally, we say that a satisfaction class $S$ is \df{regular} iff for all pairs $(\phi,\alpha) \simeq (\psi,\beta)$, $(\phi,\alpha) \in S$ iff $(\psi,\beta) \in S$.

	Before describing the relation between the interpretations of the truth predicate and regular satisfaction classes, let us introduce one more definition.

\begin{definition} \label{def_int_for_sat} \hfil
	\begin{itemize} 
		\item If $M \models \PA$ and $S\subset M^2$ is a satisfaction class, we say that \df{internal induction} holds for $\phi \in \form^{\leq 1}_{\LPA}(M)$ if for every $v \in \FV(\varphi)$ and every $\alpha \in \Asn(\phi)$, the following holds:
		\begin{displaymath}
			(\phi,\alpha[0/v]) \in S \wedge \forall x \Big((\phi,\alpha[x/v])\in S \rightarrow (\phi,\alpha[x+1/v])\in S\Big) \rightarrow \forall x \ (\phi,\alpha[x/v]) \in S. 
		\end{displaymath}
	Above, $\alpha[y/v]$ denotes the assignment $\alpha'$ which is identical to $\alpha$, except for the fact that $\alpha'(v) = y$. 
		\item We say that internal induction holds for $S$ if it holds for every $\phi \in \form^{\leq 1}_{\LPA}(M)$.
	\end{itemize}
	Notice that the formula $\phi$ need not be in the domain of $S$
\end{definition} 

The following proposition establishes the link between truth predicates satisfying $\CT^-$ (possibly with $\INT$) and regular satisfaction classes (possibly with internal induction). It can be proved via a direct verification.

\begin{proposition} \label{prop_regular_satisfaction_equiv_ctminus} \hfil
	\begin{itemize}
		\item[(a)] Let $(M,T) \models \CT^-$ and let $S = \set{(\phi,\alpha) \in M^2}{\phi \in  \form_{\LPA}(M), \alpha \in \Asn(\phi), \phi[\alpha] \in T}.$ Then $S$ is a full regular satisfaction class.
		
		\item[(b)] Conversely, let $M \models \PA$, let $S\subset M^2$ be a full regular satisfaction class in $M$ and let $T = \set{\phi \in \Sent_{\LPA}(M)}{(\phi,\emptyset) \in S}$. Then $(M,T) \models \CT^-$.
		\item[(c)] Let $(M,T) \models \CT^- + \INT$ and let $S$ be defined as in (a). Then $S$ is a full regular satisfaction class and internal induction holds for $S$.
		\item[(d)] Conversely, let $S$ be a full regular satisfaction class in $M$ such that internal induction holds for $S$. Let $T$ be defined as in (b). Then $(M,T) \models \CT^- + \INT$. 
	\end{itemize}
\end{proposition}
Now we are ready to prove Theorem \ref{th_true_infinite_disjunctions}. \\

\begin{proof}[Proof of Theorem \ref{th_true_infinite_disjunctions}, general idea.]
	
	Let $M \models \PA$. We will find an elementary extension $M \preceq M'$ and a full regular satisfaction class $S$ on $M'$ with internal induction  such that for every disjunction $\phi \in \Sent_{\LPA}(M)$ with nonstandardly many disjuncts, $(\phi,\emptyset) \in S$. Then by Proposition \ref{prop_regular_satisfaction_equiv_ctminus}, there exists $T \subset M'$ such that $(M',T) \models \CT^- + \INT$ and all disjunctions with nonstandardly many disjuncts are in $T$.

We will construct $(M',S)$ in stages. We will produce a sequence $(M_n,S_n)$ of models such that:
\begin{itemize}
	\item $M_0 := M, S_0 = \emptyset$.
	\item The models $M_i \models \PA$ form an elementary chain. 
	\item $S_{n+1}\subset M_{n+1}^2$ is a regular satisfaction class whose domain contains $\form_{\LPA}(M_n)$.
	\item For every $n$, $S_n \subset S_{n+1}$. 
	\item For every $n$, the model $M_n$ expanded with all the predicates $S_{\phi}, \phi \in \form_{\LPA}(M_n)$ defined by $S_{\phi}(x) \equiv S_n(\phi,x)$ satisfies the full induction scheme.
	\item If $\phi:= \bigvee \phi_i$ is a disjunction with nonstandardly many disjuncts and $\phi$ is in the domain of $S_{n+1}$,
	 then $(\phi,\alpha) \in S_{n+1}$ for every $\alpha \in \Asn(\phi)$.
\end{itemize} 
Finally, we set $M' = \bigcup M_n$ and $S = \bigcup S_n$. By a straightforward verification, we check that $S$ is indeed a full regular satisfaction class and it clearly makes all disjunctions with infinitely many disjuncts true and satisfies internal induction.

To complete the proof, it is enough to check that a sequence $(M_n,S_n)$ as above can be produced. This will be demonstrated in Lemma \ref{lem_induction_step_in_ev_chain} which takes care of the induction step for $n>1$ (the existence of $(M_1,S_1)$ can be proved with a very similar, slightly simpler argument).
\end{proof}

Some care is needed in order to make sure that the satisfaction classes which we will construct are indeed regular. Before we state and prove the induction step lemma, we will introduce one more technical notion. (It appeared in the same context earlier, e.g. in \citep{lelyk_wcislo_local_collection}.)

\begin{definition} \label{def_syntactic_template} 
	Let $\phi \in \form_{\LPA}$. By a \df{syntactic template} of $\phi$, we mean the smallest formula $\widehat{\phi}$ such that the following conditions are satisfied:
	\begin{enumerate}
		\item There exists a sequence of arithmetical terms $\bar{s}$ (not necessarily closed) such that $\phi = \widehat{\phi}(\bar{s})$. 
		\item No variable occurs in $\widehat{\phi}$ both free and bound.
		\item No free variable occurs in $\widehat{\phi}$ more than once.
		\item No closed term occurs in $\widehat{\phi}$.
		\item No complex term containing only free variables occurs in $\widehat{\phi}$.
	\end{enumerate}
\end{definition} 
For instance, if $\phi = \exists x (SSx+Sy = (z \times (y + S0)) \times x)$, then 
\begin{displaymath}
	\widehat{\phi} = \exists x (SSx + v_0 = v_1 \times x),
\end{displaymath}
where $v_0, v_1$ are chosen so as to minimise the formula. 

We say that $\phi$ and $\psi$ are \df{syntactically similar} if they have the same syntactic template. We denote it with $\phi \sim \psi$. Notice that if $(\phi,\alpha) \simeq (\psi,\beta)$ for some $\alpha, \beta$, then  $\phi \sim \psi$. 

\begin{lemma} \label{lem_induction_step_in_ev_chain}
	Let $M \models \PA$, let $S \subset M^2$ be a regular satisfaction class such that:
	\begin{itemize}
		\item The model $(M,S_{\phi})_{\phi \in M}$ satisfies full induction, where  $S_{\phi}(x)  \equiv S(\phi,x)$.
		\item If $\phi$ is a disjunction with nonstandardly many disjuncts in the domain of $S$, then $(\phi,\alpha) \in S$ for all $\alpha \in \Asn(\phi)$. 
	\end{itemize}
	Then there exists an elementary extension $M \preceq M'$ and a regular satisfaction class $S' \supseteq S$ such that $(M',S')$ satisfies the above conditions and $\dom(S') \supseteq \form_{\LPA}(M)$. 
\end{lemma}

\begin{proof}
	Let $M, S$ be as in the assumptions of the lemma. Let us consider a theory $\Theta$ in a language $\LPA$ extended with an additional predicate $S'$ and a family of auxiliary predicates $S'_{\phi}, \phi \in M$ which comprises the following axioms:
	\begin{itemize}
		\item $\ElDiag(M)$, the elementary diagram of $M$.
		\item $\forall x \ S'_{\phi}(x) \equiv S'(\phi,x)$ where $\phi \in \form_{\LPA}(M)$. (The definition of $S'_\phi$).
		\item $\Comp(\phi)$ for $S'$, where $\phi \in \form_{\LPA}(M)$. (Compositionality Scheme).
		\item $S'(\phi,\alpha)$, where $(\phi,\alpha) \in S$. (Preservation Scheme).
		\item $\forall \phi, \psi \in \form_{\LPA} \forall \alpha \in \Asn(\phi), \beta \in \Asn(\psi) \ \Big((\phi,\alpha) \simeq (\psi,\beta) \rightarrow S'(\phi,\alpha) \equiv S'(\psi,\beta)\Big).$ (Regularity Axiom)
		\item The full induction scheme for formulae in the language  $\LPA + S'_{\phi}, \phi \in \form_{\LPA}(M)$. (Internal Induction).
		\item $\forall \alpha \in \Asn(\phi) \ S'(\phi,\alpha)$, where $\phi = \bigvee_{i \leq c} \phi_i$ for some $(\phi_i)_{i \leq c} \in \formSeq_{\LPA}(M)$ and a nonstandard $c$. (Disjunction Scheme).
	\end{itemize}
	The predicates $S'_{\phi}$ are introduced to concisely express what form of internal induction we accept. Note that in the Internal induction scheme, we do not allow formulae containing the predicate $S'$. Crucially, we are not allowed to treat $\phi$ like a variable. On the other hand, the form of induction we accept is \textit{prima facie} stronger than the internal induction condition introduced in Definition  \ref{def_int_for_sat}.

	If $(M',S',S'_{\phi})_{\phi \in M} \models \Theta$, then by restricting $S'$ to $(\phi,\alpha)$ such that $\phi \sim \phi'$ for some $\phi' \in M$ and ignoring the predicates $S'_{\phi}$, we obtain a model satisfying the conclusion of the lemma (note that $S'$ itself need not be a satisfaction class, as the compositional conditions may possibly fail badly for formulae in $M' \setminus M$). Therefore, it is enough to check that $\Theta$ is consistent. We will prove it by compactness. Let $\Theta_0 \subset \Theta$ be a finite subtheory. It is enough to check the consistency of $\Theta_0$. 
	
	Let $\phi_1, \ldots, \phi_n \in \form_{\LPA}(M)$ be all formulae such that the instances of  their compositionality, preservation, and disjunction schemes are in $ \Theta_0$. We will define $S' \subset M^2$ which satisfies these instances of schemes, as well as all the other axioms in $\Theta$. Note that $\ElDiag(M)$ will be automatically satisfied, since the interpretation of the relation  $S'$ will be defined in $M$.  
	
	Let us consider the classes $[\phi_i]_{\sim}$, where $\sim$ is the syntactic similarity relation  (see Definition \ref{def_syntactic_template} and the remark below the definition). Let $\lhd$ be a relation defined on the classes as follows: $[\phi] \lhd [\psi]$ iff there exist $\phi' \in [\phi], \psi' \in [\psi]$ such that $\phi'$ is a direct subformula of $\psi'$. We define a sequence of relations $S_0, S_1, \ldots, S_k \subset M^2$ by induction on the rank of classes in the relation $\lhd$. 	We let $(\phi,\alpha) \in S_0$ if $[\phi]$ is minimal in the relation $\lhd$ and one of the following holds:
	\begin{itemize}
		\item There exist $s,t \in \Term_{\LPA}(M)$ such that $\phi = (s=t)$ and $s^{\alpha} = t^{\alpha}$. 
		\item $(\phi,\alpha) \in S$.
		\item $\phi$ is a disjunction with nonstandardly many disjuncts.
	\end{itemize}
	Above, we want to define $S_0$ on all classes minimal with respect to the relation $\lhd$.  If $S$ already holds of $\phi$ and $\alpha$, then we preserve it in $S_0$. If $\phi$ happens to be a disjunction with infinitely many disjuncts, then we make it true under all valuations. Otherwise, we set $(\phi,\alpha) \notin S_0$ for all $\alpha$, but we do not have to mention it explicitly.
	
	We define $S_{i+1}$ as the union of $S_i$ and the pairs $(\phi,\alpha)$ such that $[\phi]$ has rank $i+1$ and one of the following holds:
	\begin{itemize}
		\item There exists $\psi \in \form_{\LPA}(M)$ such that $\phi = \neg \psi$ and $(\psi,\alpha) \notin S_i$.
		\item There exist $\psi, \eta \in \form_{\LPA}(M)$ such that $\phi = \psi \vee \eta$ and $(\psi,\alpha) \in S_i$ or $(\eta,\alpha) \in S_i$.
		\item There exist $\psi, \eta \in \form_{\LPA}(M)$ such that $\phi = \psi \wedge \eta$ and $(\psi,\alpha) \in S_i$ and $(\eta,\alpha) \in S_i$.
		\item There exists $\psi \in \form_{\LPA}(M), v \in \Var(M)$ such that $\phi = \exists v \psi$ and $(\psi,\beta) \in S_i$ for some $\beta \sim_v \alpha$.
		\item  There exists $\psi \in \form_{\LPA}(M), v \in \Var(M)$ such that $\phi = \forall v \psi$ and $(\psi,\beta) \in S_i$ for all $\beta \sim_v \alpha$.
	\end{itemize}
	In other words, we extend $S_i$ to $S_{i+1}$ so as to satisfy the compositional clauses. Finally, since we consider only finitely  many classes, they can only attain some finite rank $k$. Let $S' = S_k$. Let  $S'_{\phi}, \phi \in M$ be defined so that the definition-axiom of  $S'_{\phi}$ is satisfied. We claim that $(M,S', S'_{\phi})_{\phi \in M} \models \Theta.$
	
	It is obvious that $(M,S')$ satisfies $\ElDiag(M)$. It follows directly by construction that it satisfies the compositional clauses for the formulae $\phi_1, \ldots, \phi_n$. It satisfies the instances of the preservation scheme for these formulae: for the formulae of minimal rank it follows by the definition of $S_0$, since if a disjunction with infinitely many disjuncts is in the domain of $S$, then it is satisfied under all assignments. For formulae of higher rank, this follows by construction, since $S$ is compositional on its domain and the compositional clauses uniquely determine the extension of a satisfaction predicate on a formula, given its extension on the direct subformulae. 
	
	Let us verify that $S'$ satisfies the regularity axiom. For formulae in the $\lhd$-minimal classes, this follows by construction and the assumption that $S$ is regular. For formulae in the classes of rank $>0$, we can directly check that compositional clauses preserve regularity and thus show by induction that all the predicates $S_i$ satisfy the regularity axiom. 
	
	The model $(M,S'_{\phi})_{\phi \in M}$ satisfies the full induction scheme, since we can verify by an easy induction on $i$ that $S_i$ is arithmetically definable in terms of  $S_{\phi_i}$, $i \leq n$, such that the class $[\phi_i]$ is $\lhd$-minimal and $\phi_i \in \dom S$. By assumption, $(M,S_{\phi})_{\phi  \in M}$ satisfies the full induction scheme and since $S'$ is definable in that structure, it satisfies induction. The conclusion follows. 
	
	The predicate $S'$ satisfies the disjunction scheme for formulae $\phi_1,\ldots,\phi_n$. Indeed, if $\phi$ is a disjunction with a nonstandard number of disjuncts and $[\phi]$ is minimal in the relation $\lhd$, then $(\phi,\alpha) \in S'$ by construction. Let us check by induction that the disjunction scheme is satisfied for formulae whose classes have higher rank. If the rank of $[\phi]$ is $i+1$, then there exist formulae $\psi, \eta$ such that $\psi$ is a disjunction with a nonstandard number of disjuncts and $\phi = \psi \vee \eta$, where the rank of $[\psi]$ is $\leq i$. By induction hypothesis, $(\psi,\alpha) \in S'$ for all $\alpha \in \Asn(\psi)$, hence by the compositional clauses $(\psi \vee \eta,\beta) \in S'$ for all $\beta \in \Asn(\psi \vee \eta)$. This concludes the proof.
\end{proof} 

\begin{remark} \label{rem_sigma_n_correctness}
	As we have already pointed out, Enayat--Visser methods of building satisfaction classes typically allow us to combine various conservativeness results. For instance, in Theorem \ref{th_true_infinite_disjunctions}, we could additionally require that the constructed predicate is correct with respect to blocks of existential or universal quantifiers.
	
	It might appear that since a class we build in the proof of this theorem is somewhat pathological, there are certain obvious limits to what truth-theoretic principles can be additionally satisfied. For instance, if the constructed truth predicate makes true all infinite disjunctions, then it cannot agree with $\Sigma_n$ arithmetical truth predicates on sentences from the respective syntactic classes. However, as pointed out by Ali Enayat, we could still construct a model in which the truth predicate $T$ agrees with all the usual partial arithmetical $\Sigma_n$-truth predicates and $\DCin$ is satisfied. It is enough that in the proof of Lemma \ref{lem_induction_step_in_ev_chain} we change the definition of $S_0$ so that for a minimal class $[\phi]$, where $\phi$ is a $\Sigma_n$-formula for some standard $n$ which is not in the domain of $S$, $S_0$ agrees with the partial arithmetical satisfaction predicates. An alternative approach, mixing an Enayat--Visser approach and resplendence (by building a satisfaction class from recursively saturated partial classes) was proposed in an unpublished note by \cite{enayat_dc_intro}.
\end{remark}

\section{Alternative versions of  $\DC$} \label{sect_parentheses}

In our paper, we defined $\bigvee_{i \leq c} \phi_i$ as	$(( \ldots (\phi_0 \vee \phi_1) \vee \ldots )\vee \phi_{c-1}) \vee \phi_c$. This is not the only possible definition of disjunction (and admittedly not the most natural one) and one could wonder whether modifying that definition has some impact on the presented results.

This point has been partially addressed in the definition of outer disjunctions: our proof of nonconservativity of $\CT^- + \DCout$ used the clause 
\begin{equation} \label{equat_first_clause_outer_disj} \tag{*}
	TD(\phi_1, \ldots, \phi_c, \phi_{c+1}) \equiv TD(\phi_1,\ldots, \phi_c) \vee T\phi_{c+1}.
\end{equation}
It turns out that this assumption cannot be weakened much further. Let us define \df{balanced disjunctions} of a sequence of formulae $\phi_1, \ldots, \phi_c$, $B(\phi_1, \ldots, \phi_c)$ by recursion in the following way: 
\begin{eqnarray*}
	B(\emptyset) & : = & 0 \neq 0 \\
	B(\phi_1) & := & \phi_1 \\
	B(\phi_1, \ldots, \phi_{c}) & := & B(\phi_1, \ldots, \phi_{\llcorner c/2 \lrcorner})  \vee B(\phi_{\llcorner c/2 \lrcorner +1}, \ldots, \phi_c).
\end{eqnarray*}
In other words, we are grouping disjunctions of $\phi_1, \ldots, \phi_k$ so that the parentheses form a binary tree. Over $\CT^-$,  balanced disjunctions do not satisfy the condition \ref{equat_first_clause_outer_disj} and in fact, if we add an analogue of $\DCout$ for that kind of disjunctions, we do not gain any arithmetical strength. 

\begin{theorem} \label{th_balanced_dc_out_not_conservative}
	$\CT^-$ together with the axiom $\forall \bar{\phi} \in \SentSeq_{\LPA} \Big(T B (\bar{\phi}) \rightarrow \exists i \leq \lh(\bar{\phi}) \ T \phi_i \Big)$ is a conservative extension of $\PA$.
\end{theorem}

\begin{proof}[Sketch of a proof.]
	Analogously to the proof of Theorem \ref{th_dcin_plus_int_conservative}, we fix a model $M \models \PA$ and we construct an elementary extension $M' \succeq M$ and $T \subset M'$ such that any balanced disjunction of a sequence $(\phi_1,\ldots, \phi_c) \in \SentSeq_{\LPA}(M')$ of nonstandard length is false.  
	
	It is enough to observe that if $\phi$ is a balanced disjunction of nonstandardly many sentences, then there exist $\phi_1, \phi_2$ such  that $\phi = \phi_1 \vee \phi_2$ and $\phi_i$ are themselves such balanced disjunctions, so we can maintain the requirement that all sentences of this form are false throughout the whole construction. 
\end{proof}

Notice that in the model produced in the proof of Theorem \ref{th_dcin_plus_int_conservative}, all balanced disjunctions with nonstandardly many disjuncts are in fact true. This shows that both the analogues of $\DCin$ and $\DCout$ for the balanced formulae are conservative over $\PA$ and that in general  $\DCout$ is sensitive to the specific kind of forming disjunctions for which they are formulated. This leads to a natural question whether some natural forms of disjunction, say, balanced disjunctions, can satisfy full $\DC$ while remaining conservative. It turns out that $\DC$ behaves stably under such varying implementations. 

\begin{theorem} \label{th_invariant_dc}
	Suppose that $U$ is a theory extending $\CT^-$. Suppose that there exists a provably functional formula $D(x)$ such that provably in $U$ for any $\bar{\phi} = (\phi_1,\ldots, \phi_c) \in \SentSeq_{\LPA}$, $D(\bar{\phi})$ is in $\Sent_{\LPA}$ and the following holds:
	\begin{displaymath}
		TD(\bar{\phi}) \equiv \exists i \leq c \ T\phi_i.
	\end{displaymath}
	Then $U$ extends $\CT_0$.
\end{theorem}
\begin{proof}
	Under the assumptions of the theorem, $D$ is an outer disjunction provably in $U$. The conclusion follows by Proposition \ref{prop_outer_disjunctions}.
\end{proof}

\begin{remark} \label{rem_selective_disj}
	As we have noted above, the results on $\DCout$ are sensitive to how we exactly define disjunctions over finite sets of sentences. One way to express abstractly what properties of disjunctions are used is the notion of outer disjunction. Ali Enayat has proposed a further generalisation of this concept. 
	
Let $U$ be a theory extending $\CT^-$. 	We say that a provably functional formula $D$ is a \df{selective disjunction} if there exists a formula $F$ which provably in $U$ defines a choice function for finite sets and the following conditions hold provably in $U$:
\begin{itemize}
	\item For any $\Phi \in \SentSet_{\LPA}$, $D(\Phi) \in \Sent_{\LPA}$.
	\item $T D(\Phi) \equiv T F(\Phi) \vee T D(\Phi \setminus \{F(\phi)\}).$ 
\end{itemize}
One can check that if a theory $U$ has selective disjunctions satisfying the analogue of  $\DCout$, namely:
\begin{displaymath}
	TD(\Phi) \rightarrow \exists \phi \in \Phi \ T\phi,
\end{displaymath}
then $U$ is again equivalent to $\CT_0$. 

\end{remark}

\section{The role of the regularity axiom} \label{sec_regularity}

In Section \ref{sec_prelim}, we briefly mentioned that we have certain technical reasons to include Regularity, $\REG$ among the axioms for the compositional truth. Let us now explain why we adopt this principle and how this choice affects our results. 

 First of all, without $\REG$, we have to be careful about the exact formulation of the compositional axioms for quantifiers. Our basic choice is between two options. We can require that, say, an existential statement $\exists v \phi(v)$ is true iff some sentence obtained by substituting a numeral in $\phi(v)$ is true. This amounts to adopting quantifier axioms in the  form which we have chosen in this article:
 \begin{displaymath}
 	T \exists v \phi(v) \equiv \exists x \ T \phi(\num{x}).
 \end{displaymath}
  The second option is to require that an existential statement is true iff a sentence obtained by substituting some closed arithmetical term for the quantified variable is true. Then, the quantifier axioms would have a form:
 \begin{displaymath}
 	T \exists v \phi(v) \equiv \exists t \in \ClTerm_{\LPA} \ T\phi(t).
 \end{displaymath} 
  Note that these two versions of the quantifier axiom   are not immediately comparable with respect to their strength. The left-to-right implication is stronger in the numeral version:
 \begin{displaymath}
 	T \exists v \phi(v) \rightarrow \exists x T \phi(\num{x})
 \end{displaymath}
 whereas the reverse implication is formally stronger for the term version:
 \begin{displaymath}
 	\exists t \in \ClTerm_{\LPA} \ T \phi(t) \rightarrow T \exists v\phi(v).
 \end{displaymath}
 
  The term version was chosen, for instance, by \cite{halbach}. The numeral version appears, e.g., in \citep{friedman_sheard} or \citep{horsten_leigh_truth_is_simple}. This choice, in  turn, may have a significant bearing on our results. Most importantly, we do not know whether adding $\Delta_0$-induction to a version of $\CT^-$ whose compositional axioms for quantifiers involve terms results in a non-conservative extension of $\PA$. So, when regularity is missing, some of the main results in this article will depend on a technical choice in the formulation of compositional axioms which may lead to some further confusions. 

Another issue is that without regularity, there is a mismatch between the notion of a truth predicate as discussed in the philosophical literature and the notion of satisfaction classes, as discussed in the literature on models of $\PA$ (e.g., as in \citep{kaye}). There is a direct link between the compositional truth and satisfaction classes assuming that we include some form of extensionality conditions in both these cases. (A more comprehensive discussion of why these kind of assumptions are relevant can be found in \citep{wcislo_definability_automorphisms}.) This link is important, as some techniques used in our proofs are designed to work specifically in the context of satisfaction classes. Most importantly, the conservativity arguments using the Enayat--Visser technique which we employ in Section \ref{sect_dcin} do not really work when directly applied to truth classes. Therefore, we would most likely need to add regularity assumptions to statements of some technical lemmas which could make for a potentially awkward reading.

Finally, let us discuss the impact of $\REG$ on the results in this article. The main results in Section \ref{sect_dcin}, namely Theorems \ref{th_dcin_plus_int_conservative} and \ref{th_true_infinite_disjunctions} are also true for $\CT^-$ without the regularity axiom, as the latter versions are formally weaker. The same applies to Remark \ref{rem_sigma_n_correctness},  Proposition \ref{prop_qfc_conservative}, and Theorem \ref{th_balanced_dc_out_not_conservative}. 

 The situation of the results in Section \ref{sec_dcout} is slightly more complicated. As we have already mentioned, Theorem \ref{th_many_faces} is only known to hold if we consider a version of $\CT^-$ in which the numeral variant of the compositional quantifier axioms is assumed. In this case, all results from Section \ref{sec_dcout}, still hold. Unfortunately, this is not necessarily true for the term variant. Let us discuss in some detail what can be salvaged in this scenario. 

Theorem \ref{th_dcout_implies_sind} does not depend at all on quantifier axioms and thus it holds without assuming regularity in either version of $\CT^-$; similarly for Theorem \ref{th_seqind_implies_dcin}. The first part of Theorem \ref{th_dc_implies_seqoind} and Theorem \ref{th_sind_has_outer_disjunction} also hold true. It is clear that $\SeqOInd$ implies $\PropRef$ for $\CT^-$ in the term version as well. Moreover, $\PropRef$ clearly implies $\DC$ using just compositional axioms for boolean connectives.  However, it is unclear whether $\PropRef$ and $\CT_0$ are equivalent without assuming $\REG$ and with the quantifier axioms in the term version. Similarly, the first part of Proposition \ref{prop_outer_disjunctions} does not depend of regularity or the quantifier axioms, but we do not know whether such theories contain $\CT_0$. Proposition \ref{prop_qfc_plus_prop_has_outer_dijunctions} may be entirely false if we do not assume the numeral version of the quantifier axioms, as the proof relies significantly on that assumption. 

To sum it up: by the results of Section \ref{sec_dcout}, $\DCout, \DC, \SeqInd$, $\SeqOInd$, and $\PropRef$ are pairwise equivalent over any version of $\CT^-$ without assuming regularity. They are equivalent to $\CT_0$ and not conservative over $\PA$ if we consider $\CT^-$ with quantifier axioms for numerals. If we consider $\CT^-$ with the term variant of the quantifier axioms, we do not know either whether $\DC$ or $\PropRef$ are equivalent to $\CT_0$, or whether the latter theory is arithmetically stronger than $\PA$. 

\section*{Appendix: a glossary of formalised notions}
Throughout the article, we referred to a number of formalised notions and used some rather technical notation. Let us now gather it in a glossary for the convenience of the reader. 

\begin{itemize}
	\item $\alpha \sim_v \beta$ means that $\dom(\beta) \supseteq \dom(\alpha) \cup \{v\}$ and $\beta(w) = \alpha(w)$ for every $w \in \dom(\alpha) \setminus \{v\}$.
	\item $\val{t}$. If $t$ is (a G\"odel code of) a closed arithmetical term, then $\val{t}$ is the value of a term (whose G\"odel code is) $t$. We use the same expression to denote the formalised version of this function.
	\item $t^{\alpha}$. If $t$ is (a G\"odel code of) an arithmetical term, and $\alpha$ is a $t$-assignment, then $t^{\alpha}$ is the value of term $t$ under the assignment $\alpha$. We use the same expression to denote the formalised version of this function.  
		\item $\val{\bar{t}}$. If $\bar{t}$ is a sequence of (G\"odel codes of) closed arithmetical terms, then $\val{\bar{t}}$ is the sequence of their values. We also use this expression to denote the formalised version of this function.
	\item $\num{x}$ is (a G\"odel code of) a canonical numeral denoting the number $x$. We also use this expression to denote the formalised version of this function.
	\item $\phi[\alpha]$. If $\phi$ is a formula and $\alpha \in \Asn(\phi)$, then by $\phi[\alpha]$ we mean a sentence obtained by substituting in $\phi$ the numeral $\num{\alpha(v)}$ for each variable $v$. We also use this expression to denote the corresponding formalised notion. 
	\item $\Asn(x)$ is a set of $x$-assignments, that is, functions whose domain contains the set of free variables of $x$, where $x$ is a term, a formula, or a sequence thereof. We also use this expression to denote the corresponding formalised notion.
	\item $\ClTerm_{\LPA}(x)$ is a formula expressing ``$x$ is (a G\"odel code of) a closed arithmetical term.'' (That is, a term with no free variables.)
	\item $\ClTermSeq_{\LPA}(x)$ is a formula expressing ``$x$ is a sequence of (G\"odel codes of) closed arithmetical terms.'' 
	\item $\FinSeq(x)$ is a formula expressing ``$x$ is a finite sequence of numbers.''
	\item $\form_{\LPA}(x)$ is a formula expressing ``$x$ is (a G\"odel code of) an arithmetical formula.''
	\item $\form^{\leq 1}_{\LPA}$ is a formula expressing ``$x$ is (a G\"odel code of) an arithmetical formula with at most one free variable.''
	\item $\formSeq_{\LPA}(x)$ is a formula expressing ``$x$ is a sequence of (G\"odel codes of) arithmetical formulae.'' 
	\item $\lh(s) = x$ is a formula expressing ``$s$ is a sequence and its length is $x$.''
	\item $\Pr_{\PA}(x)$ is a formula expressing ``$x$ is (a G\"odel code of) an arithmetical sentence provable in $\PA$.''
	\item $\Pr_{\Prop}(x)$ is a formula expressing ``$x$ is (a G\"odel code of) an arithmetical sentence which is provable in pure propositional logic.'' 
	\item $\Pr^{T}_{\Prop}(x)$ is a formula expressing ``$x$ is (a G\"odel code of) an arithmetical sentence which is provable in propositional logic from the set of premises $\Gamma$ such that $T(y)$ holds for all $y \in \Gamma$.''
	\item  $\qfSent_{\LPA}(x)$ is a formula expressing ``$x$ is (a G\"odel code of) a quantifier-free arithmetical sentence.''
	\item $\Sent_{\LPA}(x)$ is a formula expressing ``$x$ is (a G\"odel code of) an arithmetical sentence.''
	\item $\SentSeq_{\LPA}(x)$ is a formula expressing ``$x$ is a sequence of (G\"odel codes of) arithmetical sentences.''
	\item $\SentSet_{\LPA}(x)$ is a formula expressing ``$x$ is a finite set of (G\"odel codes of) arithmetical sentences.'' 
	\item $\Term_{\LPA}(x)$ is a formula expressing ``$x$ is (a G\"odel code of) an arithmetical term.''
	\item $\Tr_0$ is the arithmetical truth  predicate for $\Delta_0$-formulae.
	\item $\Var(x)$ is a formula expressing ``$x$ is (a G\"odel code of) a first-order variable.''

\end{itemize}

\section*{Acknowledgements}
We are very grateful to Ali Enayat for a number of helpful comments and suggestions which allowed us to improve this article. 

This research was supported by an NCN MAESTRO grant 2019/34/A/HS1/00399 ``Epistemic and Semantic Commitments of Foundational Theories.''

\end{document}